\documentclass[a4paper,draft]{ZbRad}
\usepackage{amssymb}
\usepackage{amscd} 

\theoremstyle{plain}
 \newtheorem{thm}{Theorem}[section]
 \newtheorem{prop}{Proposition}[section]
 \newtheorem{lem}{Lemma}[section]
 \newtheorem{cor}{Corollary}[section]
\theoremstyle{definition}
 \newtheorem{exm}{Example}[section]
 \newtheorem{rem}{Remark}[section]
 \newtheorem{dfn}{Definition}[section]
\numberwithin{equation}{section}

\renewcommand{\le}{\leqslant}
\renewcommand{\ge}{\geqslant}
\renewcommand{\setminus}{\smallsetminus}
\newcommand{\mathset}[1]{{\left\{#1\right\}}}
\newcommand{\absolute}[1]{\left\lvert#1\right\rvert}
\newcommand{\norm}[1]{\left\|#1\right\|}

\DeclareMathOperator{\GL}{GL}
\DeclareMathOperator{\Star}{Star}
\DeclareMathOperator{\kNNh}{\it{k-NN_h}}
\DeclareMathOperator{\height}{ht}
\DeclareMathOperator{\supp}{supp}

\DeclareMathOperator{\closure}{cl}

\newfont{\TITf}{cmssdc10 scaled 1440}

\title[]{Boundary value problems on $p$-adic analytic manifolds}

\author[]{Patrick Erik Bradley}

\begin{document}

\setcounter{page}{1}

\vspace*{40mm}

\thispagestyle{empty}

{
\TITf\setlength{\parskip}{\smallskipamount}

\begin{center}
Patrick Erik Bradley

\bigskip\bigskip\bigskip

{Boundary Value Problems on $p$-Adic Analytic Manifolds}

\end{center}
}

\vspace{6em}{\leftskip3em\rightskip3em
\emph{Abstract}.
An account is given on newest developments on $p$-adic boundary value problems on $p$-adic analytic manifolds and their relationship with diffusion. In particular, novel coordinate Laplacians on $p$-adic analytic $n$-manifolds  constructed with the help of frame bundles, are introduced in this context. These are used to construct elliptic  operators. Related Dirichlet problems are formulated and solved, generalising results on compact subdomains of $p$-adic $n$-space.
In the end, an outlook towards  number-theoretic applications as  well as extensions of this theory to ultrametric analytic manifolds is given. This is a substantial upgrade of the presentation given at Branko's 80-th Birthday Conference in Belgrade, May 2025.



}

\newpage\thispagestyle{empty}

\maketitle
\tableofcontents

\section{Introduction}  

Stochastic processes on $p$-adic domains have been enabled via the Vladimirov operator on $p$-adic number fields and their finite products, and are also known as the Vladimirov-Taibleson operator, cf.\ \cite{VVZ1994,RZ2008}. Further developments on $p$-adic stochastic processes can be found e.g.\ in \cite{Kochubei2001,Zuniga2025}.
Diffusion on other $p$-adic domains has been studied e.g.\ in \cite{PRSW2024} ($p$-adic integers), \cite{PW2025} ($p$-adic vector spaces), 
\cite{RodriguezDiss,Rodriguez2025} (certain $p$-adic Lie groups).
Diffusion on compact $p$-adic analytic manifolds using an atlas is first defined in \cite{DiffMfp}, and applied to hearing their Serre invariant \cite{HearingSerre}. 
\newline

Dirichlet problems on $p$-adic number fields are studied in 
\cite{Kochubei2023}. Elliptic operators on compact subdomains of such are developed, and associated boundary value problems are studied in \cite{ellipticBVP}. Ongoing work also aims at extending the results here to ultrametric manifolds, in order to generalise \cite{VPZeta}.
\newline

The theory presented here is developped in more detail in the ongoing habilitation project \cite{diff_AMf_p}.

\section{Diffusion on $p$-adic analytic manifolds}

\subsection{Basic constructions}

Let $K$ be a $p$-adic number field, i.e.\ a finite extension of the field $\mathbb{Q}_p$ of $p$-adic numbers. The absolute value on $K$ is written here as $\absolute{\cdot}$, and the maximum norm on the vector space $K^n$ is
\[
\norm{x}_{K^n}=\max\mathset{\absolute{x_1},\dots,\absolute{x_n}}
\]
with $x=(x_1,\dots,x_n)\in K^n$.
The ring of integers of $K$ is denoted as $O_K$, and an element $\pi\in O_K$ with
\[
\absolute{\pi}=1
\]
is a \emph{uniformiser} of $K$.
For those unexperienced with such a general $p$-adic field $K$, it suffices to think of $K=\mathbb{Q}_p$. In this case, $O_K=\mathbb{Z}_p$, i.e.\ the ring of $p$-adic integers, and take $p=\pi$ in this case.
\newline

The Haar measure on $K^n$ is denoted as $\mu$, or as $\absolute{dx}$, where the latter is usually used in the context of integration, e.g.\ as
\[
\int_Kf(x)\,\absolute{dx}=\int_Kf(x)\,d\mu(x)=\int_Kf(x_1,\dots,x_n)\,\absolute{dx_1}\wedge\dots\wedge\absolute{dx_n}
\]
using the product Haar measure. It is normalised as
\[
\mu(O_K^n)=1\,,
\]
and, more precisely, 
\[
\mu_i(O_K)=1
\]
for the component measures of
\[
\mu=\absolute{dx_1}\wedge\dots\wedge\absolute{dx_n}
\]
on each copy of $K$ making up the vector space $K^n$.
\newline

The space $\mathcal{D}(X)$ on a topological space $X$ is defined as the set of functions $X\to\mathbb{C}$ which are locally constant with compact support.

\begin{dfn}
A function $F\colon U\to K^n$, where $U\subseteq K^n$ is an open subset,
is called $K$-analytic, if it is locally given by a convergent power series with $n$
variables and coeﬃcients in $K$.
\end{dfn}

An \emph{atlas} $\mathcal{A}$ on a Hausdorff space $X$ is a collection of compatible charts on $X$. The latter is a pair $(U,\phi)$, where $U\subseteq X$
is open, and $\phi\colon U\to K^n$ is a homeomorphism onto its image. The compatibility property means that for each pair of charts $(U_\alpha,\phi_\alpha)$, $(U_\beta,\phi_\beta)\in\mathcal{A}$, it
holds true that the transition function
\[
\tau_{\alpha\beta}\colon\phi_\alpha(U_\alpha\cap U_\beta)\to\phi_\beta(U_\alpha\cap U_\beta),\; x\mapsto \phi_\beta(\phi^{-1}_\alpha(x))
\]
is one-to-one, and its inverse $\tau_{\alpha\beta}^{-1}$ coincides with the transition function $\tau_{\beta\alpha}$.
In the case of an analytic manifold, the transition maps (and their inverses) are
all locally bi-analytic maps. In other words, they are for $x \in U_\alpha\cap U_\beta$ given in
a small neighbourhood of $\phi_\alpha(x)\in\phi_\alpha(U_\alpha\cap U_\beta)$ as convergent power series
in $n$ variables with coeﬃcients in $K$.

\begin{dfn}
A \emph{$p$-adic analytic manifold} is a pair $(X,\mathcal{A})$, where $X$ is a
Hausdorff topological space, and $\mathcal{A}$ an atlas on $X$ such that all transition maps
between overlapping charts are $K$-analytic functions. Its dimension is $n\in\mathbb{N}$, if
for all charts $(U,\phi)\in\mathcal{A}$, the map $\phi$ takes $U$ to $K^n$. In this case, $(X,\mathcal{A})$ is also called a \emph{$p$-adic analytic $n$-manifold}.
\end{dfn}

Two atlantes $\mathcal{A}$ and $\mathcal{B}$ on $X$ are \emph{equivalent}, if their union $\mathcal{A}\cup\mathcal{B}$ is again
an atlas on $X$. This is readily seen to be an equivalence relation on the set of
all atlantes of $X$. An equivalence class of a given atlas $\mathcal{A}$ is called a \emph{$p$-adic
analytic structure} on $X$. From the books \cite{Schneider2011,Serre1992}, one can learn more about about
$p$-adic analytic manifolds. Here, it is of interest to identify in a set $X$ with a
given $p$-adic analytic structure $[\mathcal{A}]$ a certain type of atlantes within $[\mathcal{A}]$.

\begin{rem}
It is assumed that $p$-adic analytic manifolds are paracompact, as then they are Polish spaces, and thus amenable to stochastic processes.
\end{rem}

Given an analytic differential $n$-form $\omega$ on a $p$-adic analytic $n$-manifold $(X,\mathcal{A})$, one obtains in a natural way a measure on $X\setminus V(\omega)$, where $V(\omega)\subset X$ is the vanishing locus of $\omega$. This construction is explained in \cite[Chapter 7.4]{Igusa2002} and in \cite[Chapter 2.2]{Weil1982}, and goes as follows: locally, on a chart $(U,\phi)$, the differential form is given as
\[
\omega|_U=\phi_*\omega=f_U\,dx=f_U\,dx_1\wedge\dots\wedge dx_n\,,
\]
where $f_U\colon \phi(U)\to K$ is a $K$-analytic function. Then
\[
\mu_{\omega|_U}(A)=\int_{\phi(U)}\absolute{f_U(x)}\,\absolute{dx}=\int_{\phi(U)}\absolute{f_U(x_1,\dots,x_n)}\absolute{dx_1}\wedge\dots\wedge\absolute{dx_n}\,,
\]
where $A\subset U$ is a Borel set, and 
$\absolute{dx_i}$ is the normalised Haar measure on $K$, and
\[
\absolute{dx}=\absolute{dx_1}\wedge\dots\wedge\absolute{dx_n}
\]
is the product Haar measure on $K^n$. These local measures glue together to a measure $\mu_\omega=\absolute{\omega}$ on $X$, such that
\[
\mu_\omega(A)=\int_X\absolute{\omega}=\int_{X}d\mu_\omega(x)
\]
for $A\subset X$ a Borel set. That this construction is well-defined can be shown via the transformation rule for integration which reflects in absolute value the transformation rule for differential forms.

\subsection{Tangent Bundle on a $p$-Adic Manifold}

The brief construction of the tangent bundle on a real manifold, as can be found e.g.\ in \cite[Chapter 1.1]{Taira1998}, carries over to the $p$-adic case as follows:
\newline

Let $(X,\mathcal{A})$ be a $p$-adic analytic manifold. The following relation is an equivalence relation: the set consists of triples $(U,\phi,v)$ with $(U,\phi)\in\mathcal{A}$, and $v\in K^n$, and the relation $(U,\phi,v)\sim(V,\psi,w)$ has the meaning
\[
(\psi\circ\phi^{-1})'(\phi(x))(v)=w\,,
\]
where $(\psi\circ\phi^{-1})'$ is the derivative of 
\[
\psi\circ\phi^{-1}\colon \phi(U\cap V)\to\psi(U\cap V)\,,
\]
which is a $K$-analytic map for the overlap $U\cap V$ of the charts $(U,\phi),(V,\psi)$ in their purpose of being local neighbourhoods of a fixed point $x\in X$.
The \emph{tangent space} $T_x(X)$ in $x$ is the set of equivalence classes for this relation $\sim$.
It can be endowed with the structure of a $K$-vector space by assigning a class $[(U,\phi,v)]$ to
$v\in K^n$, and then transporting the vector space structure on $K^n$ via this in fact well-defined bijection.
\newline

Just like in the real case, the tangent bundle is defined via the individual tangent spaces $T_x(X)$ of a $p$-adic analytic $n$-manifold $(X,\mathcal{A}(X))$, whose atlas is denoted here as $\mathcal{A}(X)$, as follows:
\[
TX=T(X)=\bigsqcup\limits_{x\in X}T_x(X)
\]
with the projection map
\[
\pi_X\colon T(X)\to X,\;T_x(X)\ni\bar{v}\mapsto x\,.
\]
Let $(U_\alpha,\phi_\alpha)\in\mathcal{A}(X)$. Then define
\[
\tau_\alpha\colon\pi^{-1}_X(U_\alpha)\to\phi_\alpha(U_\alpha)\times
K^n,\;T_x(X)\ni\bar{v}\mapsto(\phi_\alpha(x),v)
\]
with $x\in U_\alpha$, 
and transition map
\[
\tau_\alpha\circ\tau_\beta^{-1}\colon
\phi_\alpha(U_\alpha\cap U_\beta)\times K^n\to\phi_\beta(U_\alpha\cap U_\beta)\times K^n
\]
via
\[
(\phi_\alpha(x),v)\mapsto\left(\phi_\beta(x),\left(\phi_\beta\circ\phi_\alpha^{-1}\right)'(\phi_\alpha(x))v\right)
\]
for $x\in U_\alpha\cap U_\beta$, $v\in K^n$. 

\begin{lem}
The set $T(X)$ together with the atlas $\mathcal{A}(T(X))$ consisting of the charts $(\pi^{-1}(U_\alpha),\tau_\alpha)$, where $(U_\alpha,\phi_\alpha)\in\mathcal{A}(X)$, is a 
$p$-adic analytic $2n$-manifold.
\end{lem}

\begin{proof}
This follows from the fact  the Jacobian matrix $\left(\phi_\beta\circ\phi_\alpha^{-1}\right)'$ is a bi-analytic isomorphism at $\phi_\alpha(x)$ for any $x\in X$.
\end{proof}

\subsection{Frame bundles}

The tangent bundle associated with a $p$-adic analytic $n$-manifold is given by the projection map
\[
\pi_X\colon T(X)\to X\,,
\]
where $T(X)$ is the tangent bundle of $X$. This is an example of a vector bundle on the manifold $X$. A  \emph{$K$-analytic frame} in $x\in X$ w.r.t.\ the tangent bundle is given by the assignment
\begin{align}\label{analyticFrame}
x\mapsto (b_1(x),\dots,b_n(x))\,,
\end{align}
where each $(b_1(x),\dots,b_n(x))$ is a basis of the tangent space $T_x(X)$ in $x\in X$. This can be made explicit by pulling back the standard basis of $K^n$ via the isomorphism
\[
T_x(X)\to K^n
\]
in each fibre of $\pi_X$ in $x\in X$, such that the resulting assignment (\ref{analyticFrame}) is a $K$-analytic map into the space $\mathcal{F}(TX)$ of frames, which itself is defined as the underlying space of the fibre bundle
\[
\mathcal{F}(TX)\to X\,,
\]
given as a so-called \emph{analytic principal $\GL_n(K)$-bundle} without going into the details of its definition, here. What is important here, is that each $b_i(x)$ is the $i$-th column of a matrix
\[
(b_1(x),\dots,b_n(x))\in\GL_n(K)
\]
for $x\in X$ and $i=1,\dots,n$.
\newline

The existence of $K$-analytic frames for compact $p$-adic analytic manifolds $X$ lies in the triviality of the tangent bundle $T(X)$:
\[
T(X)=X\times K^n\,,
\]
which follows from Serre's Theorem that $X$ is $K$-analytically isomorphic with the disjoint union of finitely many copies of $O_K^n$, cf.\ \cite[Th\'eor\`eme (1)]{Serre1965}. 

\begin{rem}
The triviality of the tangent bundle of a compact $p$-adic analytic manifold $X$ also follows inherently from the existence of a nowhere vanishing $K$-analytic differential form on $X$. This is in stark contrast to the case of a compact algebraic variety, where (in general) algebraic differential forms do have zeros, and this is an obstruction to the triviality of the associated (algebraic) tangent bundles.
\end{rem}

\subsection{Integral Structures}

\begin{dfn}
An integral structure on an $n$-dimensional $K$-vector space
is a finitely generated $O_K$-submodule 
$L\subset V$ which spans $V$ as a $K$-vector space.
\end{dfn}

\begin{dfn} Let $X$ be a $p$-adic analytic $n$-manifold. An integral structure $\Lambda$ on $X$ is the assignment $x\mapsto\Lambda_x$ of an integral structure on each tangent
space $T_x(X)$, which is locally constant in the sense that for each chart $(U, \phi)$ and
$z \in U$, there is an integral structure $L_z\subset K^n$ such that
\[
T_x \phi(\Lambda_x) = L_z
\]
for all $x$ in a neighbourhood of $z$.
\end{dfn}

\begin{dfn}
An atlas $\mathcal{A}$ on a $p$-adic analytic $n$-manifold is said to be $O_K$-
compatible, if for all $(U,\phi), (V, \psi) \in\mathcal{A}$ and $x \in U \cap V$ it holds true that the
derivative $T_x \tau_{UV} =\tau'_{UV}$ of the transition function
\[
\tau_{UV}\colon\phi(U) \to \psi(V)
\]
satisfies $T_x \tau_{UV}\in\GL_n(O_K)$.
\end{dfn}

\begin{lem} Any compact $p$-adic analytic manifold $X$ with an integral structure has a finite $O_K$-compatible atlas $\mathcal{A}$.
\end{lem}

\begin{proof}
Since, according to the proof of \cite[Lemma 3.2.6]{BKL2026}, every integral structure on $X$ gives rise to an $O_K$-compatible atlas, and due to the compactness
of $X$, the assertion is immediate.
\end{proof}

According to \cite[Section 3.3]{BKL2026}, an integral structure $\Lambda$ on a $p$-adic analytic manifold $X$ induces a Radon measure on $X$ as follows: The integral structure $\Lambda$ induces a unique integral structure on $\bigwedge^nT_x(X)$ for $x\in X$, such that the map
\[
X\to\bigwedge^nT_x(X),\;
x\mapsto\omega_X(x)\,,
\]
where $\omega_X(x)$ is a generator of the $1$-dimensional vector space $\bigwedge^nT_n(X)$ over $K$,
is locally constant. Thus, $O_K\omega_X(x)\subset\bigwedge^nT_x(X)$ defines an integral structure on the exterior power $\bigwedge^nT_x(X)$. Since the differential $n$-form $\omega_X(x)$ defining this integral structure is unique up to a multiplicative constant factor in $O_K^\times$, it follows that the modulus
\[
\absolute{\omega_X(x)}
\]
is uniquely determined by the integral structure $\Lambda$ on $X$. In this way, a unique measure $\mu_\Lambda$ is defined via
\[
\int_X f\,d\mu_\Lambda=\int_Xf(x)\absolute{\omega_X(x)}
\]
for functions $f\colon X\to\mathbb{C}$, and in particular
\[
\mu_\Lambda(A)=\int_X1_A\,d\mu_\Lambda=\int_A\absolute{\omega_X}\,,
\]
using the indicator function $1_A$
for a Borel set $A\subset X$.

\begin{dfn}
The Radon measure $\mu_\Lambda$ is called a \emph{canonical measure} on $X$. 
\end{dfn}

Observe that a canonical measure $\mu_\Lambda$ is always nowhere vanishing on $X$.

\begin{rem}
Since a $K$-analytic differential form $\omega$ on $X$ leads to a locally constant map
\[
X\to\bigwedge^n T_x(X),\;x\mapsto \omega(x)\,,
\]
and this can be used to construct an integral structure on $X$, this now means that
a $p$-adic analytic $n$-manifold always has a canonical measure, because it has a nowhere vanishing differntial $n$-form according to \cite[Th\'eor\`eme (2)]{Serre1965}.
\end{rem}

\subsection{Compact $p$-Adic Analytic Manifolds as Posets}

Given an open covering $\mathcal{U}$ of a topological space $X$, its associated \emph{nerve complex} $N(\mathcal{U})$ is defined as follows:
\begin{align*}
U&\neq\emptyset,\quad U\in\mathcal{U}&\text{(vertices)}
\\
U\cap V&\neq\emptyset,\quad U,V\in\mathcal{U}&\text{(edges)}
\\
&\vdots&\vdots\quad
\\
U_0\cap\dots\cap U_k&\neq\emptyset,\quad U_0,\dots,U_k\in\mathcal{U}\;\text{pairwise disjoint}&\text{($k$-simplices)}
\\
\vdots&&\text{(etc.)}\,.
\end{align*}
For any given atlas $\mathcal{A}$, it will be assumed that any two distinct charts $(U,\phi),(V,\psi)\in\mathcal{A}$ have the property that their underlying sets $U,V$ are also distinct. In this case, it makes sense to speak of the \emph{nerve complex} $N(\mathcal{A})$ of the atlas $\mathcal{A}$, defined as the nerve complex of the covering of the manifold $X$ underlying $\mathcal{A}$.
\newline

Let now $X$ be a compact $p$-adic analytic manifold endowed with an integral structure $\Lambda$, and a finite $O_K$-compatible atlas $\mathcal{A}$, safisfying the condition for having a well-defined nerve complex $N(\mathcal{A})$ as a simplicial complex. A further assumption made here is that the simplicial complex $N(\mathcal{A})$ is connected. This is an important assumption in the light of Serre's Theorem about compact $p$-adic analytic manifolds being disjoint unions of $p$-adic balls, and the connectedness of $N(\mathcal{A})$ can be ensured via taking suitable overlapping unions of such balls as constituents of an atlas $\mathcal{A}$. 

\begin{lem}\label{equalising}
The transition functions of an $O_K$-compatible atlas $\mathcal{A}$ take $p$-adic balls to $p$-adic balls of equal radius.
\end{lem}

\begin{proof}
The transition maps $\tau_{UV}$ for overlapping charts $(U,\phi),(V,\psi)$ of $\mathcal{A}$ satisfy $\tau_{UV}'\in\GL_n(O_K)$. Hence, locally in $x\in U\cap V$, 
\[
\tau_{UV}'O_K^n=O_K^n\,,
\]
and thus $\tau_{UV}$ takes $O_K^n$ to a translate of the ball $O_K^n$ in $K^n$, which is a $p$-adic ball of equal radius $1$ as $O_K^n$ is. The power series expansion
\[
\tau_{UV}(T)=a_0+\tau_{UV}'(0)T+\text{h.o.t.}
\]
with $\tau_{UV}'(0)\in\GL_n(K)$ being the constant term in the power series expansion of $\tau_{UV}'(T)$,
now yields
\[
\tau_{UV}(\pi^kO_K^n)=a_0+\pi^kO_K^n
\]
for $k>0$. Since any $p$-adic ball in $K^n$ is a translate of some $\pi^kO_K^n$ with $k\in\mathbb{N}$, the assertion now follows.
\end{proof}

This fact allows to define the notion of a ball on a compact $p$-adic analytic manifold $(X,\mathcal{A})$ by defining a set $B\subset X$ to be a ball of radius $r>0$, if for any  chart $(U,\phi)$, with $U$ containing $B$, the image $\phi(B)$ is a $p$-adic ball in $K^n$ of radius $r>0$. The following exception is to hold by convention: if a $k$-simplex for $\mathcal{A}$ happens to be ball-shaped after being taken into $K^n$ by a chart map, it will nevertheless \emph{not} be called a ball in $X$. The reason is that the distinction ``face'' or ``simplex'' on the one hand, and ``ball'' on the other hand needs to be upheld in order to avoid confusion where this distinction matters. Using this convention, it is also the case for $p$-adic analytic manifolds with integral structure that $p$-adic balls never overlap. 

\begin{exm}
Let $X=\mathbb{P}^1(\mathbb{Q}_2)$ with the standard atlas 
\[
\mathcal{A}=\mathset{(U_0,\phi_0),(U_1,\phi_1)}
\]
with
\begin{align*}
U_0&=\mathset{[x_0:x_1]\in\mathbb{P}^1(\mathbb{Q}_2)\mid\absolute{x_0}_2=1,\;\absolute{x_1}_2\le1}\,,
\\
U_1&=\mathset{[x_0:x_1]\in\mathbb{P}^1(\mathbb{Q}_2)\mid\absolute{x_1}_2=1,\;\absolute{x_0}_2\le1}
\end{align*}
and associated chart maps
\begin{align*}
\phi_0&\colon U_0\to\mathbb{Q}_2,\;[x_0:x_1]\mapsto\frac{x_1}{x_0}\,,
\\
\phi_1&\colon U_1\to\mathbb{Q}_2,\;[x_0:x_1]\mapsto\frac{x_0}{x_1}\,.
\end{align*}
The intersection $U_0\cap U_1$ is mapped to the unit sphere
in $\mathbb{Q}_2$ under both chart maps $\phi_0,\phi_1$, with corresponding transition maps $z\mapsto z^{-1}$. But this image set happens to be a $2$-adic disc, whereas $U_0\cap U_1$ is the $1$-simplex of the nerve complex $N(\mathcal{A})$. 
\end{exm}

Define $\mathfrak{B}(X,\mathcal{A})$ as the set of balls and $k$-simplices in $X$ for $k\in\mathbb{N}$. The inclusion relation endows $\mathcal{B}(X,\mathcal{A})$ with a natural poset structure. In general, this poset is not a tree. It is the disjoint union of the finite poset of the  faces of the nerve complex $N(\mathcal{A})$ and the trees of $p$-adic balls in $X$, each  rooted in a face of $N(\mathcal{A})$, since these are all compact open $p$-adic analytic submanifolds of $X$. 
Define the poset
\[
\overline{\mathfrak{B}}(X,\mathcal{A})=\mathfrak{B}(X,\mathcal{A})\bigsqcup X
\]
as the natural poset extension of $\mathcal{B}(X,\mathcal{A})$ given by appending it with the boundary $X$ of the union of the trees of balls attached to the faces of $N(\mathcal{A})$.

\begin{dfn}
A  \emph{path} $a\leadsto b$ with $n$ points in $\mathfrak{B}(X,\mathcal{A})$ is a sequence
\[
a=x_0,\dots,x_n=b
\]
with $x_0,\dots,x_n\in\mathfrak{B}(X,\mathcal{A})$ pairwise distinct, such that for $i=0,\dots,n-1$, it holds true that either
\[
x_i< x_{i+1}
\]
or 
\[
x_{i+1}< x_i\,,
\]
where $<$ is the ordering of the poset $\mathcal{B}(X,\mathcal{A})$ which excludes the reflexivity property of a partial ordering. 

\smallskip
A \emph{path} in the tree associated to a $p$-adic ball is an infinite nested decreasing sequence of $p$-adic balls such that no $p$-adic ball fits strictly between two consecutive balls in the sequence.  

\smallskip
A \emph{path} $x\leadsto y$ in  $\overline{\mathfrak{B}}(X,\mathcal{A})$ between $x\neq y$ in $X$ is a triple
\[
(\gamma_x,\gamma^\circ_{xy},\gamma_y)\,,
\]
where $\gamma_x,\gamma_y$ are maximal paths in trees rooted in some simplices $U_x,U_y$ in $N(\mathcal{A})$ and ending in $x$, respectively $y$, $\gamma_{xy}$ is a  path $U_x\leadsto U_y$ in $\mathfrak{B}(X,\mathcal{A})$ with finitely many points, and where $x$ and $y$ are each the intersections of the balls constituing the paths $\gamma_x$ and $\gamma_y$, respectively. 
\end{dfn}

\begin{rem}
The fact that $K$ is spherically finite  guarantees that any infinite descending nested sequence of balls always has a non-empty intersection, and this intersection is a point in $K$. In other words, all paths in the trees of balls of $X$ have an endpoint contained in $X$.
\end{rem}

In order to now define a notion of geodetic distance, let
$U_{\mathfrak{s}}$ for $\mathfrak{s}\in\mathfrak{B}(X,\mathcal{A})$ be either a simplex of $N(\mathcal{A})$ represented by the point $\mathfrak{s}\in\mathfrak{B}(X,\mathcal{A})$ contained in the poset of faces associated with $N(\mathcal{A})$, or
a ball in $X$ represented by the point $\mathfrak{s}$ contained in the poset associated with a tree of balls rooted in a face of $N(\mathcal{A})$. The reason for this distinctive language will become apparent in the next paragraph.
\newline

First, define
\[
U_{\mathfrak{s}}^\circ(x)=U_{\mathfrak{s}}\setminus\bigcup\limits_{\mathfrak{t}<\mathfrak{s}\atop x\not<\mathfrak{t}}U_{\mathfrak{t}}
\]
for $x\in X$ and $\mathfrak{s}\in\mathfrak{B}(X,\mathcal{A})$ such that $x<\mathfrak{s}$. If $\gamma\colon x\leadsto y$ is a path between $x,y\in X$ inside $\overline{\mathfrak{B}}(X,\mathcal{A})$, then write $\mathfrak{s}\in\gamma$ for any point $\mathfrak{s}$ in $\mathfrak{B}(X,\mathcal{A})$ lying in the path $\gamma$.
Now, define its \emph{length} $\ell(\gamma)$ as
\begin{align}\label{pathlength}
\ell_\Lambda(\gamma)=\sum\limits_{\mathfrak{s}\in\gamma_x}\mu_\Lambda(U_{\mathfrak{s}}^\circ(x))+\sum\limits_{\mathfrak{s}\in\gamma_y}\mu_\Lambda(U_{\mathfrak{s}}^\circ(x))
+\sum\limits_{\sigma\in\gamma_{xy}^\circ}\mu_\Lambda(U_\sigma)\,.
\end{align}
The \emph{geodetic distance} between $x\neq y$ in $X$ is defined as
\[
d_\Lambda(x,y)=\inf\mathset{\gamma_\Lambda(\gamma)\mid\gamma\colon x\leadsto y\;\text{is a path between $x$ and $y$ in $\mathfrak{B}(X,\mathcal{A})$}}\,,
\]
and depends on the integral structure $\Lambda$ on $X$.

\begin{lem}
The geodetic distance $d_\Lambda(x,y)$ extends to a metric on $X$.
\end{lem}

\begin{proof}
The extension is evidently to be taken as $d_\Lambda(x,x)=0$ for $x\in X$. If $x\neq y$, then any path $\gamma\colon x\leadsto y$ contains an element of $\mathfrak{B}(X,\mathcal{A})$. Since the nerve complex is finite, and its complement in $\mathfrak{B}(X,\mathcal{A})$ is a disjoint union of finitely many rooted trees, it follows that there are only finitely many paths between $x$ and $y$, and each has finite length. Hence, the infimum is a positive number, and thus $(x,y)\mapsto d_\Lambda(x,y)$ is a positive definite function.  

\smallskip
The symmetry of $d_\Lambda(x,y)$ is immediate.

\smallskip
The triangle inequality follows thus: first,
\[
d_\Lambda(x,y)=\ell_\Lambda(\gamma_x)+d_\Lambda(\sigma(x),\sigma(y))+\ell_\Lambda(\gamma_y)\,,
\]
where $\gamma_x$ is the geodesic in the rooted tree of balls containing $x$ with root $\sigma(x)$ in the nerve complex $N(\mathcal{A})$. Likewise, $\gamma_y$ is the geodesic connecting simplex $\sigma(y)$ with $y$. Further, $d_\Lambda(\sigma(x),\sigma(y))$ is the third summand in the definition of $\ell_\Lambda(\gamma)$ in (\ref{pathlength}). And $\gamma=(\gamma_x,\gamma_{xy}^\circ,\gamma_y)$ is a minimising path for the geodetic distance whose part in $N(\mathcal{A})$ connects $\sigma(x)$ with $\sigma(y)$. It exists by the argument in the beginning of the proof. Let now $z\in X$ be a third point. Then 
\begin{align*}
d_\Lambda(x,y)&\le\ell_\Lambda(\gamma_x)+d_\Lambda(\sigma(x),\sigma(z))+d_\Lambda(z,y)+\ell_\Lambda(\gamma_y)
\\
&\le\ell_\Lambda(\gamma_x)+d_\Lambda(\sigma(x),\sigma(z))+2\ell(\gamma_z)+d_\Lambda(\sigma(z),\sigma(y))
\\
&=d_\Lambda(x,z)+d_\Lambda(z,y)\,,
\end{align*}
where $\sigma(z)$ is the root of the tree of balls containing $z$, and $\gamma_z$ is the geodesic between $\sigma(z)$ and $z$ in that tree.

\smallskip
The assertion now follows.
\end{proof}

\subsection{Geodetic distance kernels}
Having a geodetic distance on a $p$-adic analytic manifold $X$ with integral structure $\Lambda$, allows to define a 
general kernel function
\[
w(x,y)=w(d_{\Lambda(X)}(x,y))
\]
in order to obtain an integral Laplacian
operator
\[
\Delta_wu(x)=\int_Xw(d_{\Lambda(X)}(x,y))(u(x)-u(y))\,d\mu_\Lambda(y)
\]
for $u\in\mathcal{D}(X)$, where the latter is the space of locally constant functions $X\to\mathbb{C}$.
Special cases of such an operator are the \emph{Vladimirov-Taibleson operator}
\[
\Delta^\alpha u(x)=\int_Xd_\Lambda(x,y)^{-n\alpha}(u(x)-u(y))\,d\mu_\Lambda(y)
\]
with $\alpha\in\mathbb{R}$, 
which naturally generalises the known Vladimirov-Taibleson operator on $K^n$, and is a non-local operator in the sense that its kernel function connects any two points of the manifold $X$.
\newline

Defining a more localised operator takes some preparation.
\newline

The \emph{height} of a point $\xi\in\mathfrak{B}(X,\mathcal{A})$ is defined as
\[
\height(\xi)=\text{length of a maximal chain in $\mathfrak{B}(X,\mathcal{A})$ down to $\xi$}\,,
\]
where ``down to'' means that the vertices of the nerve complex $N(\mathcal{A})$ are considered as being the maximal elements of $\mathfrak{B}(X,\mathcal{A})$, and
the points of $X$ as being the minimal elements of the extended poset $\overline{\mathfrak{B}}(X,\mathcal{A})$.
\newline

The \emph{join} of $\xi,\eta\in\mathfrak{B}(X,\mathcal{A})$ is defined as the supremum of the set $\mathset{\xi,\eta}$ in the poset $\mathfrak{B}(X,\mathcal{A})$, if it exists and is unique. 
It is denoted as
\[
\xi\wedge\eta\,.
\]
The notion of join is naturally extended to $\overline{\mathfrak{B}}(X,\mathcal{A})$. Notice that $x\wedge y$ need not be a ball in $X$, but can also be a simplex of $N(\mathcal{A})$.

\begin{dfn}
The height-$k$-nearest neighbour graphon $\kNNh(X)$
associated with $X$ has the vertex set $X$, and an edge is given by pairs 
$(x,y)\in X$ with 
$x\neq y$ and
\[
\height(x \wedge y) \ge k
\]
for fixed $k \le\dim(N(\mathcal{A}))$.
\end{dfn}

Observe that the edges of the height-$k$-nearest neighbour graphon
are undirected. Observe also that the height-$0$-nearest neighbour graphon is connected, since the nerve complex $N(\mathcal{A})$ is connected.
The \emph{degree} of $x$ in $\kNNh(X)$ is defined as
\[
\deg_k(x)=\int_{y\sim x}\mu_\Lambda(x\wedge y)\,d\mu_\Lambda(y)\,,
\]
where $x\sim y$ means that $x,y\in X$ are connected by an edge in $\kNNh(X)$. Observe that 
\[
\Star_k(x)=\mathset{y\in X\mid x\sim y}
\]
is a measurable subset of $X$, and that $\deg_k$ is locally constant on $X$.

\begin{lem}\label{degreeConst}
The graphon $\kNNh(X)$ is connected, if and only if $k\le\dim(N(\mathcal{A}))$.
If $k\in\mathbb{N}$ is sufficiently large, then  $\deg_k$ is constant on any ball of $X$.
\end{lem}

\begin{proof}
If $k\le\dim(N(\mathcal{A}))$, then $\kNNh(X)$ is clearly connected. Conversely, if $k > \dim(N(\mathcal{A}))$, then the
edges of $\kNNh(X)$ are given by pairs
$(x,y)$, where $x\wedge y$ is a ball strictly 
contained in a minimal simplex of $N(\mathcal{A})$ containing it. Consequently, since
minimal simplices of $N(\mathcal{A})$ are roots of trees in $\mathcal{B}(X,\mathcal{A})$, the connected
components of $\kNNh(X)$ are balls forming a disjoint covering of $X$. Because 
$k > \dim(N(\mathcal{A}))$, it follows that there are at least two connected components of $\kNNh(X)$. 

\smallskip
The latter assertion follows in case edges connect only points in $X$ being ends of the same tree rooted in a simplex of $N(\mathcal{A})$, and from the fact that these trees are regular. And this plays out for sufficiently large $k$ by the fact that the modulus $\absolute{\omega_X(x)}$ induced by the integral structure $\Lambda$  is locally constant on $X$.
\end{proof}

The \emph{height-$k$-nearest neighbour graphon Laplacian} $\Delta_k^\alpha$ on $X$ is defined as
\[
\Delta_{k}^\alpha u(x)=
\int_X T_{k}^\alpha(x,y)(u(x)-u(y))\,d\mu_\Lambda(y)
\]
with kernel function
\[
T_{k}^\alpha(x,y)=\begin{cases}
\frac{\displaystyle\deg_k(y)}{\displaystyle\deg_k(x)}\cdot d_\Lambda(x,y)^{-n\alpha},&y\in\Star_{k}(x)
\\
0,&\text{otherwise}
\end{cases}
\]
for $x,y\in X$, acting on functions $u\in\mathcal{D}(X)$.
\newline

A \emph{Kozyrev wavelet} supported on a ball $B(a)\subset K^n$ with $m\in\mathbb{Z}$ is a
function of the following kind:
\[
\psi_{B_m(a),j}(x)=\mu(B_m(a))^{-\frac12}\chi(\pi^{m-1}\tau(j)x)1_{B_m(a)}(x)\,,
\]
where $\chi$ is a unitary character of $K^n$, $x\in K^n$, $j\in\mathbb{F}_q^n$, $\tau\colon\mathbb{F}_q^n\to O_K^n$ is a lift, and $\mu$ is the Haar measure on $K^n$ such that $\mu(O_K^n)=1$, and where
\[
\tau(j)x=\tau_1(x_1)+\dots+\tau_n(x_n)
\]
with $j=(j_1,\dots,j_n)\in\mathbb{F}_q^n$,
$x=(x_1,\dots,x_n)\in K^n$, $\tau=(\tau_1,\dots,\tau_n)$.

\begin{dfn}
A wavelet on $(X,\mathcal{A})$ is a function $\psi\colon X\to\mathbb{C}$ supported in
a ball $B(a)\subset X$ such that for any chart $(U, \phi)$ containing $B$, the function
\[
\phi_*\psi\colon\phi(B)\to\mathbb{C}, \phi(x)\mapsto\psi(x)
\]
with $x \in B(a)$ is a rescaled Kozyrev wavelet supported on $\phi(B(a))$,   such
that $\norm{\psi}_{L^2(X),\mu}=1$.
\end{dfn}

The well-definedness of the notion of wavelet follows from the fact that a point $a\in X$ is given first
by a family of points $\phi(a)\in K^n$ associated with a family of charts $(U,\phi)$ such
that $a\in U$, since the transition maps on overlapping charts $(U,\phi)$, $(V,\psi)$
containing $a$ take $\phi(a)$ to $\psi(a)$, and, secondly by Lemma \ref{equalising}. 
A wavelet on $X$ will be denoted as
\[
\psi=\psi_{B(a),j}(x)
\]
with $a\in X$, ball $B(a)\subset X$ its support, and $j\in\mathbb{F}_q^n$.
\newline

A wavelet on a chart $(U,\phi)$ such that $U$ contains the support of
a wavelet $\psi$ on $X$ has the form
\[
\phi_{*}\psi(\phi(x)) 
= \mu_\Lambda (B(a))^{\frac12}
\chi\left(\pi^{m-1}\tau(j)\phi(x)\right)
1_{\phi(B(a))} (\phi(x))
\]
for $x\in X$, and where $\supp(\psi)=B(a)\subset X$.
\newline

Given a wavelet $\psi$ supported  on $B\subset X$, an important property is the vanishing integral:
\begin{align}\label{vanishingInt}
\int_X\psi(x)\,d\mu_\Lambda(x)=0\,,
\end{align}
which is shown in a similar way as 
\cite[Lemma 2.15]{hearGenusMumf}.
Since $\omega_X$ is, by construction, a nowhere vanishing analytic differential
$n$-form on $X$, it can locally on a chart $(U, \phi)$ containing $B$ be written  as
\[
\omega_X|_{U} (z)= g_U (\phi(z))\, dz
\]
with analytic function $g_U(z)$ defined over $O_K$. Its absolute value coincides
with that of a polynomial, non-vanishing everywhere in $U$. Hence,
\begin{align*}
\int_X\psi(x)\,&\absolute{\omega_X(x)}
=
\int_{\phi(U)}\psi_{\phi(B(a)),j}(z)\absolute{g_U(z)}\,dz
\\
&=C\int_{\phi(U)}\psi_{\phi(B(a)),j}(z)\,dz
\end{align*}
for some $C>0$, if the support $B(a)$ of $\psi$ is sufficiently small. In this case, the vanishing of the Kozyrev wavelet on any measurable
set containing its support is well known, for example, cf. 
\cite[Theorem 2.3.9]{XKZ2018}.

\begin{lem}\label{waveletSpec}
A wavelet $\psi$ supported on a sufficiently small ball $B(a)\subset X$ is an eigenfunction of the Vladimirov-Taibleson operator
 $\Delta^\alpha$ with eigenvalue
\[
\lambda_\psi=\int_{X\setminus B(a)}d_\Lambda(a,y)^{-n\alpha}\,d\mu_\Lambda(y)
+\mu_\Lambda(B(a))^{n-\alpha}
(1-q^{-n}(1+(-1)^n))
\]
for $\alpha\in\mathbb{R}$.
\end{lem}

\begin{proof}
First observe that the independence of the
choice of $x\in B(a)$ is immediate. Next, the formula is formally identically
with the one in \cite[Theorem 3]{Kozyrev2004}, and the proof of that carries over, because
in the first integral, it is not needed that $d_\Lambda(x, y)$ is an ultrametric, where
$x \in B(a)$ and $y \in X \setminus B(a)$. The proof uses only the fact that $d_\Lambda$ is an
ultrametric inside $B(a)$. Indeed, in this case, if $x\notin B(a)$, then
\begin{align*}
\Delta^\alpha\psi_{B(a),j} (x) &=
\int_X d_\Lambda(x,y)^{-n\alpha} (\psi_{B(a),j} (x) - \psi_{B(a),j} (y))\,d\mu_\Lambda(y)
\\
&= \psi_{B(a),j} (x)
\int_X d_\Lambda(x,y)^{-n\alpha} \,d\mu_\Lambda(y)
\\
&- d_\Lambda(x,B(a))^{-n\alpha}
\psi_{B(a),j}(y)\, d\mu_\Lambda (y) = 0
\end{align*}
by the vanishing property (\ref{vanishingInt}). If $x\in B(a)$, then
\begin{align*}
\Delta^\alpha\psi_{B(a),j}(x)
&=
\psi_{B(a),j}(x)
\int_{X\setminus B(a)}
d_\Lambda(x,y)^{-n\alpha} d\mu_\Lambda(y)
\\
&+
\int_{B(a)}d_\Lambda(x,y)^{-n\alpha} (\psi_{B(a),j}(x) - \psi_{B(a),j}(y))\, d\mu_\Lambda(y)\,.
\end{align*}
Obtain that the latter integral equals
\begin{align*}
I&=
\int_{\partial B(a)}
d_\Lambda(x,y)^{-n\alpha}(\psi_{B(a),j}(x)-\psi_{B(a),j}(y))\,d\mu_\Lambda(y)
\\
&=
\int_{\partial\phi(B(a))}
\norm{\phi(x) - z}_{K^n}^{-n\alpha}
(\psi_{\phi(B(a)),j}(x) - \psi_{\phi(B(a)),j}(z))
\absolute{g_U(z)}\absolute{dz}\,,
\end{align*}
where $(U,\phi)\in\mathcal{A}$ is a chart containing $B(a)$, and the canonical differential form $\omega_X$ takes on $\phi(U )$ the form $g_U(z)\,dz$. Notice again, that, since $B(a)$ is sufficiently small,
$\absolute{g_U(y)}$ is constant on the circle $\partial\phi(B(x))$. Hence, assuming that $B(x) =
B(a)$ has radius $q^{-k}$, obtain
\[
I = C\int_{\norm{\phi(x) - z}_{K^n}=q^{-k}}
\norm{\phi(x)-z}_K^{-n\alpha}
(\psi_{\phi(B(a)),j} (\phi(x)) - \psi_{\phi(B(a)),j} (z)) \absolute{dz}
\]
for some $C > 0$ and $k\in\mathbb{N}$. Points $z \in K^n$ with
\[
\norm{z-\phi(x)}_{K^n}=q^{-n}
\]
have the form
\[
z=\phi(x) + u + \epsilon
\]
with $u = (u_1,\dots,u_n) \in\pi^k O_K^n$
such that for $i\in\mathset{1,\dots,n}$:
\[
u_i = \pi^k\tau_i(\ell_i),\quad
\ell_i\in\mathbb{F}_q\,,
\]
and $\ell_i\in \mathbb{F}_q^\times$ for  at least one $i\in\mathset{1,\dots, n}$, and with $\epsilon\in\pi^{k+1}O_K$. Hence, 
\begin{align*}
I &= C
\sum\limits_{i=1}^n
\sum\limits_{\ell=(\ell_1,\dots,\ell_n)\in\atop\mathbb{F}_q\times\dots\times\mathbb{F}_q^\times\dots\times\mathbb{F}_q}
\int_{\pi^{k+1}O_K^n}
\norm{\tau(\ell)+\epsilon}_{K^n}^{-n\alpha}[\psi_{\phi(B(a)),j}(\phi(x))
\\
&-\psi_{\phi(B(a)),j}(x+\tau(\ell)+\epsilon)]\absolute{d\epsilon}
\\
&=
 Cq^{\alpha nk-n(k+1)}
\sum\limits_{\ell\in\mathbb{F}_q^n\setminus\mathset{\bar{0}}}
(\psi_{\phi(B(a)),j} (\phi(x))
- \psi_{\phi(B(a)),j} (\phi(x) + \tau(\ell)))
\\
&\stackrel{(*)}{=}
C q^{\alpha nk-n(k+1)}
\sum\limits_{\ell\in\mathbb{F}_q^n\setminus\mathset{\bar{0}}}
(1 - \chi(\pi^{-1}\tau(j)\tau(\ell))) \psi_{\phi(B(a)),j} (\phi(x))
\\
&\stackrel{(**)}{=}
C q^{\alpha nk-n(k+1)}
(q^n - (1 + (-1)^n))\psi_{\phi(B(a)),j} (\phi(x))
\\
&= 
C\mu(\phi(B(a)))^{n-\alpha}
(1- q^{-n} (1 + (-1)^n)
\psi_{\phi(B(a)),j}(\phi(x))
\\
&= \mu_\Lambda(B(a))^{n-\alpha}
( 1 - q^{-n} (1 + (-1)^n )\psi_{B(a),j} (x)\,,
\end{align*}
where $(*)$ follows from the property of Kozyrev wavelets:
\[
\psi_{B(a),j} (x + \pi^k\tau(\ell)) = \chi(\pi^{-1} \tau(j)\tau(\ell))\psi_{B(a),j} (x)
\]
for $x\in K^n$, and
$(**)$ follows from
\[
\sum\limits_{\ell\in\mathbb{F}_q^n\setminus\mathset{\bar{0}}}
(1 - \chi(\pi^{-1}\tau(j)\tau (\ell))) 
= q^n - (1 + (-1)^n)\,,
\]
which follows by induction on $n\in\mathbb{N}$. 
This now implies the asserted eigenvalue.
\end{proof}

\begin{lem}\label{waveletSpec_kNNh}
A wavelet $\psi$ supported on a sufficiently small ball $B(a)\subset X$ is an eigenfunction of the height-$k$-nearest neighbour Laplacian $\Delta_k^\alpha$ with eigenvalue 
\[
\lambda_\psi=\int_{\Star_k(a)\setminus B(a)}
\mu_\Lambda(a\wedge y)^{-n\alpha}\,d\mu_\Lambda(a)
+\mu_\Lambda(B(a))^{n-\alpha}(1-q^{-n}(1+(-1)^n))
\]
for $\alpha\in\mathbb{R}$.
\end{lem}

\begin{proof}
If the support $B(a)$ is sufficiently small, then the kernel function $T_k^\alpha(x,y)$ depends only on an ultrametric distance
on the support $B(a)$ of the wavelet, because the degree function $\deg_k(x)$ is
constant on $B(a)$, cf.\ Lemma \ref{degreeConst}. It now follows that the proof of 
Lemma \ref{waveletSpec} carries over, because the necessary vanishing integral property (\ref{vanishingInt})
holds true for the wavelet $\psi$. This proves the asserted value for the wavelet eigenvalue $\lambda_\psi$.
\end{proof}

\begin{rem}
The eigenvalue formulae in Lemmas \ref{waveletSpec} and \ref{waveletSpec_kNNh} extend the wavelet eigenvalue formula from \cite[Theorem 3]{Kozyrev2004} to more general domains, as given by $p$-adic analytic $n$-manifolds.
\end{rem}

\begin{thm}
The operator $\Delta^\alpha$ acting on the Hilbert space 
$L^2(X,\mu_\Lambda)$ is self-
adjoint, positive semi-deﬁnite, and yields an orthogonal decomposition
\[
L^2(X,\mu_\Lambda)
= W_0\oplus L^2(X,\mu_\Lambda)_w \,,
\]
and an orthonormal basis consisting of eigenfunctions of $\Delta^\alpha$ whose wavelet part
spans the summand 
$L^2(X,\mu_\Lambda)_w$, and whose orthogonal complement spans the
finite-dimensional closed subspace $W_0$, 
where $\alpha\in\mathbb{R}$. Each eigenvalue
has only finite multiplicity. For $\alpha > 0$, the spectrum of $\Delta^\alpha$ is a point spectrum.
For $\alpha\le 0$, the operator $\Delta^\alpha$ is a compact operator, and $0$ is an accumulation point
of its spectrum. 
\end{thm}

\begin{proof}
The finite-dimensional space $W_0$ is defined to be spanned by copies of indicator functions on the largest possible ball $B\subset X$ which is sufficiently small that $\deg_k$ is constant on $B$, and such that these copies are a finite disjoint covering of $X$. Such a ball exists by to Serre's theorem \cite[Th\'eor\`eme (1)]{Serre1965}. 

\smallskip
The assertions are essentially a consequence of 
Lemma \ref{waveletSpec}, and the observation that the symmetric operator $\Delta^\alpha$ restricted to $W_0$ acts as a symmetric Laplacian on a complete weighted graph. Notice that the direct sum on
the right-hand side is the full space $L^2(X,\mu_\Lambda)$, because any locally constant
function orthogonal to all wavelets must be constant on the maximal balls in
$X$ whose indicators span $W_0$. 
Hence, by density of $\mathcal{D}(X)$ in 
$L^2(X,\mu_\Lambda)$, the equality of spaces follows.
The self-adjointness of $\Delta^\alpha$ now follows, as well as positive-semidefiniteness:
the wavelet eigenvalues are all positive, as well as the graph eigenvalues, except the one corresponding to the constant functions which is zero. Finally,
the eigenvalues exhibited in
Lemma \ref{waveletSpec} clearly have only finite multiplicity. Hence, all eigenvalues of $\Delta^\alpha$ have only finite multiplicity. If $\alpha>0$,
then $\Delta^\alpha$ is a densely defined unbounded operator on $L^2(X,\mu_\Lambda)$, and its spectrum contains no accumulation point. Hence, it is a point spectrum. If $\alpha \le 0$, then compactness of $\Delta^\alpha$ follows from boundedness and the finite multiplicities of the eigenvalues. Lemma \ref{waveletSpec} shows that $0$ is an accumulation
point of the spectrum of $\Delta^\alpha$ in this case. This proves the theorem.
\end{proof}

\begin{thm}\label{kNNh_Spectrum}
The operator $\Delta^\alpha_k$ acting on the Hilbert space $L^2(X,\mu_\Lambda)$ is
diagonalisable, its spectrum is 
non-negative, and each eigenvalue has finite multiplicity. If $\alpha>0$, then its spectrum is discrete.
\end{thm}

\begin{proof}
The supports of wavelets $\psi$ on $X$ being balls, it follows from  the local constancy of the degree function $\deg_k$  that
all wavelets on $X$ with sufficiently small support form an orthonormal system in 
$L^2(X,\mu_\Lambda)$ consisting of
eigenfunctions for $\Delta_k^\alpha$. The orthogonal complement $W_k^\perp$ of the closure of
the span $W_k$ of these wavelets is finite-dimensional due to the compactness of $X$.
Since $\deg_k$ is constant on the supports of the wavelets spanning $W_k$, it follows that $W_k^\perp$ is also invariant under $\Delta_k^\alpha$, because it is spanned by the indicators of the maximal supports of these wavelets. On that space, $\Delta_k$ takes the form
\[
\Delta^\alpha_k= D_kA_k 
=: C_k\,,
\]
where $D_k$ is a diagonal matrix, and $A_k$ is a symmetric matrix. Hence, the
detailed balance property
\begin{align*}
C_k D_k = D_k C_k^\top
\end{align*}
in the terminology of 
\cite[(4.2)]{vanKampen2007} is valid. Hence, $C_k$ is diagonalisable, cf.
\cite[Chapter 7]{vanKampen2007}. The diagonalisability of $\Delta^\alpha_k$ now follows. The assertions
about the discreteness of the spectrum of $\Delta_k^\alpha$, and the finiteness of eigenvalue
multiplicities also follow due to Lemma \ref{waveletSpec_kNNh}. By that lemma, the wavelet eigenvalues are
positive, and the eigenvalues of the restriction of $\Delta_k^\alpha$ to $W_k^\perp$
are non-negative, because that operator is a graph Laplacian. The discreteness of the spectrum of $\Delta_k^\alpha$ for $\alpha>0$ also follows from Lemma \ref{waveletSpec_kNNh}.
\end{proof}

\section{Boundary Value Problems}

\subsection{Markov property}

The \emph{Sobolev space} $W^{1,2}_H(X,\mu_\Lambda)$ with respect to an operator $H$ on $L^2(X,\mu_\Lambda)$ is defined as
\[
W^{1,2}_H(X,\mu_\Lambda)=\mathset{f\in L^2(X,\mu_\Lambda)\mid \norm{H f}_{L^2(X,\mu_\Lambda)}<\infty}
\]
and is endowed with the Sobolev norm
\[
\norm{f}_{W^{1,2}_H(X,\mu_\Lambda)}
=\left(
\norm{f}_{L^2(X,\mu_\Lambda)}^2+
\norm{Hf}_{L^2(X,\mu_\Lambda)}^2
\right)^{\frac12}
\]
for $f\in W^{1,2}_H(X,\mu_\Lambda)$.

\begin{lem}
The Sobolev space $W^{1,2}(X,\mu_\Lambda)$ is a Hilbert space w.r.t.\ the Sobolev norm $\norm{\cdot}_{W^{1,2}(X,\mu_\Lambda)}$.
\end{lem}

\begin{proof}
Define
\[
\langle f, g\rangle_{W^{1,2}_H(X,\mu_\Lambda)} = 
\langle f, g\rangle_{L^2(X,\mu_\Lambda)} + \langle Hf,Hg\rangle_{L^2(X,\mu_\Lambda)}\,,
\]
where
\[
\langle u, v\rangle_{L^2(X,\mu_\Lambda)}
=\int_X
u(x)\overline{v(x)}\,d\mu_\Lambda(x)
\]
for $u,v \in W^{1,2}_H(X,\mu_\Lambda)$. This is readily seen as an inner product on $W^{1,2}_H(X,\mu_\Lambda)$.
The proof of completeness follows a standard approach, and will be omitted here.
\end{proof}

\begin{prop}[Poincar\'e Inequality]\label{PoincareIneq}
Assume that the operator is diagonalisable on $L^2(X,\mu_\Lambda)$, and all eigenvalues are non-negative and unbounded. Assume further that there exists an orthonormal basis of $L^2(X,\mu_\Lambda)$ consisting of eigenfunctions of $H$ which belong to $W^{1,2}_H(X,\mu_\Lambda)$.
Let $u\in W^{1,2}_H(X,\mu_\Lambda)$. Then 
\[
\norm{u}_{L^2(X,\mu_\Lambda)}\le C\norm{u}_{W^{1,2}_H(X,\mu_\Lambda)}
\]
for some $C>0$.
\end{prop}

\begin{proof}
The expansion 
\[
u=\sum\limits_{\psi}\alpha_\psi\psi
\]
for $u\in W^{1,2}_H(X,\mu_\Lambda)$, where $\psi\in W^{1,2}_H(X,\mu_\Lambda)$ are taken from an orthonormal basis of $L^2(X,\mu_\Lambda)$ consisting of eigenfunctions of $H$. Due to the unboundedness of the eigenvalues $\lambda_\psi$ corresponding to $\psi$, it follows that
\[
\norm{u}_{L^2(X,\mu_\Lambda)}
=\sum\limits_{\psi} \absolute{\alpha_\psi}^2 \le C \sum\limits_\psi\absolute{\alpha_\psi}^2\absolute{\lambda_\psi}^2
=C\norm{H u}^2_{L^2(X,\mu_\Lambda)}\,,
\]
as asserted.
\end{proof}

\begin{lem}
The semigroup $e^{-tH}$ acts compactly on $L^2(X,\mu_\Lambda)$ for $t>0$, if
$\alpha > 0$.
\end{lem}

\begin{proof}
If $\alpha > 0$, 
then the operator $e^{-tH}$ acting on 
$L^2(X,\mu_\Lambda)$ is of trace-class
for $t > 0$. Hence, it operates compactly on that Hilbert space.
\end{proof}

The spectrum of a linear operator $A$ is often denoted as $\sigma(A)$. Its complement $\rho(A)$ in $\mathbb{C}$ is the resolvent set of $A$. The resolvent of an operator $A$ is
defined as the operator
\[
R(\lambda, A) = (\lambda - A)^{-1}\,,
\]
where $\lambda\in\rho(A)$.

\begin{thm}[Hille-Yosida, contraction case]\label{HilleYosidaContraction}
 Let $A$ be a linear operator on
a Banach space $E$. The following statements are equivalent:
\begin{enumerate}
\item $A$ generates a strongly continuous contraction semigroup.
\item $A$ is closed, densely defined, and for every $\lambda > 0$, one has $\lambda\in\rho(A)$ and
\[
\norm{\lambda R(\lambda, A)}\le 1\,.
\]
\item
$A$ is closed, densely defined, and for every $\lambda\in\mathbb{C}$ with $\Re(\lambda)>0$, one has
$\lambda\in\rho(A)$ and
\[
\norm{R(\lambda, A)}\le\frac{1}{\Re(\lambda)}
\]
\end{enumerate}
\end{thm}

\begin{proof}
Cf.\ \cite[Theorem II.3.5]{EK2000}.
\end{proof}

These results are applied here to the Vladimirov operator $\Delta^\alpha$.

\begin{thm}\label{VTMarkov}
Let $\alpha > 0$. 
The operator $\Delta^\alpha$ is the generator of a strongly continuous semigroup $e^{-t\Delta^\alpha}$ on $W_{\Delta^\alpha}^{1,2}(X,\mu_\Lambda)$ for $t \ge0$, which satisfies the Markov
property.
\end{thm}

The proof follows the argumentation in the proof of \cite[Theorem 6.3]{ellipticBVP}.

\begin{proof}
Notice first of all, that $-t\Delta^\alpha$ is of trace class according to Lemma \ref{waveletSpec}. Hence, it acts compactly on $L^2(X,\mu_\Lambda)$ for $t\ge0$.
Since  its
eigenvalues are bounded from above, as they are non-positive according to
Lemma \ref{waveletSpec},
it follows that $e^{-t\Delta}$ is a strongly continuous semigroup acting on 
$L^2(X,\mu_\Lambda)$.

\smallskip
The space $W^{1,2}_{\Delta^\alpha}(X,\mu_\Lambda)$ is clearly invariant under the action of $e^{-t\Delta^\alpha}$ for
$t \ge 0$. The contraction semigroup property of $e^{-t\Delta^\alpha}$ acting on $W_{\Delta^\alpha}^{1,2}(X,\mu_\Lambda)$
follows, because
\begin{align}\label{importantIneq}
\norm{\int_0^t e^{-\tau\Delta^\alpha u}\,d\tau}_{W_{\Delta^\alpha}^{1,2}(X,\mu_\Lambda)}\le t\norm{u}_{W_{\Delta^\alpha}^{1,2}(X,\mu_\Lambda)}\,,
\end{align}
which can be checked first for $L^2$-normalised eigenfunctions $\psi$ with eigenvalue $\lambda_\psi$:
\begin{align*}
\norm{\int_0^t e^{-\tau\Delta^\alpha}\psi\,d\tau}^2_{W^{1,2}_{\Delta^\alpha}(X,\mu_\Lambda)}&
=\norm{\int_0^te^{-\tau\lambda_\psi}\psi\,d\tau}^2_{L^2(X,\mu_\Lambda)}
\\
&+\norm{\Delta^\alpha\int_0^t e^{-\tau\lambda_\psi}\psi\,d\tau}^2_{L^2(X,\mu_\Lambda)}
\\
&=\left(\int_0^t e^{-\tau\lambda_\psi}\psi\,d\tau\right)^2
(1+\lambda_\psi^2)
\\
&=\frac{(1-e^{-t\lambda_\psi})^2}{\lambda_\psi^2}(1+\lambda_\psi)^2
\\
&\le t^2(1+\lambda_\psi)^2
=t^2\norm{\psi}^2_{W_{\Delta^\alpha}^{1,2}(X,\mu_\Lambda)}\,,
\end{align*}
because
\[
\frac{1-e^{-t\lambda_\psi}}{\lambda_\psi}\le t
\]
for $t\ge 0$. Then, for general 
$u \in W_{\Delta^\alpha}^{1,2}(X,\mu_\Lambda)$, 
this follows by expanding $u$ into
a sum of eigenfunctions, and then applying Pythagoras.

\smallskip
The reason why (\ref{importantIneq}) implies the contraction semigroup property can be
shown as explained by User Mizar on StackExchange Mathematics on Jan.\
19th, 2015 at 23:54 like this: 
using that
\[
R(\lambda,\Delta^\alpha)u =\lambda
\int_0^\infty e^{-\lambda t}
\int_0^te^{-\tau\Delta^\alpha}\,d\tau\,dt
\]
is an expression for the resolvent
\[
R(\lambda,\Delta^\alpha) = (\lambda-\Delta^\alpha)^{-1}\,,
\]
it follows that
\begin{align*}
\norm{R(\lambda,\Delta^\alpha)u}_{W^{1,2}_{\Delta^\alpha}(X,\mu_\Lambda)}
&\le\int_0^\infty e^{-\lambda t} 
\norm{\int_0^t e^{-\tau\Delta^\alpha}u\,d\tau}_{W^{1,2}_{\Delta^\alpha}(X,\mu_\Lambda)}\,dt
\\
&\le \lambda \int_0^\infty e^{-\lambda t}t\norm{u}_{W^{1,2}_{\Delta^\alpha}(X,\mu_\Lambda)}\,dt
\\
&=\frac{1}{\lambda}\norm{u}_{W^{1,2}_{\Delta^\alpha}(X,\mu_\Lambda)}\,,
\end{align*}
implying that
\begin{align}\label{ResolventIneq}
\norm{R(\lambda,\Delta^\alpha)}\le\frac{1}{\lambda}\,,
\end{align}
and so, the Hille-Yosida Theorem for contraction semigroups (Theorem \ref{HilleYosidaContraction})
shows that $e^{-t\Delta}$ with $t\ge0$ is a contraction semigroup on $W^{1,2}_{\Delta^\alpha}(X,\mu_\Lambda)$: Namely, the spectrum of $-t\Delta^\alpha$ being negative implies the condition $\lambda>0$ $\Rightarrow $ $\lambda\in\rho(-t\Delta^\alpha)$. From the unitary diagonalisability and (\ref{ResolventIneq}), the other conditions of Hille-Yosida follow.

\smallskip
What remains is the Markov property. It
follows first from showing that
\begin{align}\label{positive}
f \ge 0\; \text{a.e.}&\Rightarrow e^{-t\Delta^\alpha} f \ge 0\;\text{a.e.}
\\\label{lessThanOne}
f \le 1\; \text{a.e.}&\Rightarrow e^{-t\Delta^\alpha} f \le 1\; \text{a.e.}
\\\label{indicator}
e^{-t\Delta^\alpha}1_X &= 1_X\,,
\end{align}
and then exhibiting an invariant measure for $e^{-t\Delta^\alpha}$ for $t\ge0$.

\smallskip
Statements (\ref{positive}) and (\ref{lessThanOne}) follow like this: If $f$ is real-valued, then $f$ is a
linear combination
of eigenfunctions which is invariant under the action of
the torus $\left(\mathbb{F}_q^\times\right)^n$ on each chart $(U,\phi)$ via the map
\[
x = (\xi_1,\dots,\xi_n) \mapsto \tau(j)x = (\tau_1 (j_1)\xi_1,\dots,\tau_n(j_n)\xi_n),
\]
where $j = (j_1,\dots,j_n)\in\left(\mathbb{F}_q^\times\right)^n$
and a map
\[
\tau\colon\mathbb{F}_q^n\to O_K
\]
which takes each element of $\mathbb{F}_q^n$ to a unique representative modulo $\pi O_K$
which takes each element of $\mathbb{F}_q^n$ to a unique representative modulo $\pi O_K$.
Now, if $f\ge 0$, then it is a positive and torus-invariant linear combination
of eigenfunctions of $\Delta^\alpha$. This is then also the case for $e^{-t\Delta^\alpha}$. If $f\ge 1$, then
again, $e^{-t\Delta^\alpha}$ is real-valued, and also $\ge 1$, because all eigenvalues are non-
negative by Lemma \ref{waveletSpec}. Property (\ref{indicator}) follows from the fact that $1_X$ is an
eigenfunction of $\Delta^\alpha$ with eigenvalue $0$.

\smallskip
An invariant measure is found thus: on each eigenspace $E_\psi$, the observation
\[
e^{-t\Delta^\alpha_\psi} := 
e^{-t\Delta^\alpha}|_{E_\psi} 
= e^{-t\lambda_\psi}\,,
\]
leads to an invariant measure $\pi_\psi$ such that
\[
e^{-t\Delta_\psi^\alpha}\pi_\psi f_\psi(x)
=\int_X f_\psi(y)\,d\pi_\psi(y)\,,
\]
whose existence follows from the symmetry of the operator on $E_\psi$. Write
$f\in\mathcal{D}(X)$ as a finite sum
\[
f=\sum\limits_\psi f_\psi\,,
\]
where $f_\psi$ is the orthogonal projection of $f$ onto $E_\psi$. Then the formal direct
sum
\[
\pi_{W^{1,2}}=\sum\limits_\psi\pi_\psi
\]
satisfies
\[
e^{-t\Delta^\alpha}\pi_{W^{1,2}}f
=\sum\limits_{E_\psi}
e^{-t\Delta^\alpha_\psi}\pi_\psi f_\psi
=\sum\limits_{E_\psi}\int_X f_\psi\,d\pi_\psi
=\int_X f\,d\pi_{W^{1,2}}
\]
as a distribution on $\mathcal{D}(X)$. In order to exhibit $\pi_{W^{1,2}}$ also as a distribution on
$W^{1,2}(X,\mu_\Lambda)$, approximate $f\in W^{1,2}(X,\mu_\Lambda)$ with a convergent sequence of
test functions $f^{(m)} \in\mathcal{D}(X)$, and observe that
\[
\sum\limits_{E_\psi}e^{-t\Delta_\psi^\alpha}\pi_\psi f^{(m)}
=\sum\limits_{E_\psi}\int_X f_\psi^{(m)}\,d\pi_\psi
=e^{-t\Delta^\alpha}\pi_{W^{1,2}}f^{(m)}
\]
converges for $m\to\infty$ to
\begin{align}\label{invariantDist}
\int_X f\,d\pi_{W^{1,2}}
=\sum\limits_{E_\psi}\int_X f_\psi\,d\pi_\psi
= \sum\limits_{E_\psi} e^{-t\Delta^\alpha_\psi}\pi_\psi f_\psi
= \left(\sum\limits_{E_\psi}e^{-t\Delta_\psi^\alpha}\pi_\psi\right)f\,,
\end{align}
and this converges, because $\pi_\psi\in E_\psi$ is a multiplier measure of the form 
$\lambda_\psi$,
and thus $f \in W^{1,2}(X,\mu_\Lambda)$ satisfies
\begin{align*}
\absolute{\int_X f\,d\pi_{W^{1,2}}}^2
&=\absolute{\langle f,\pi_{W^{1,2}} \rangle}^2
=\sum\limits_{E_\psi} \absolute{\langle f_\psi,\pi_{W^{1,2}}}^2
=\sum\limits_{E_\psi} \lambda_\psi^2\norm{f_\psi}^2_{E_\psi}
\\
&\le \norm{f}^2_{W^{1,2}_{\Delta^\alpha}(X,\mu_\Lambda)}\,.
\end{align*}
This means that
\[
\sum\limits_{E_\psi}e^{-t\Delta^\alpha_\psi}\pi_\psi\in W^{1,2}_{\Delta^\alpha}(X,\mu_\Lambda)'
\]
is a distribution on $W^{1,2}(X,\mu_\Lambda)$ which coincides with the formally given
distribution
\[
e^{-t\Delta^\alpha}\pi_{W^{1,2}}
\]
together with the identity (\ref{invariantDist}). Hence, $\pi_{W^{1,2}}$ is the distribution on $W^{1,2}_{\Delta^\alpha}(X,\mu_\Lambda)$
invariant under $e^{-t\Delta^\alpha}$ for $t > 0$. This proves the assertions.
\end{proof}

\begin{rem}
The proof of Theorem \ref{VTMarkov} requires restricting to the Sobolev
space $W^{1,2}_{\Delta^\alpha}(X,\mu_\Lambda)$ in order to prove the existence of an invariant distribution
for $e^{-t\Delta^\alpha}$ for $t > 0$.
\end{rem}
\subsection{$\kNNh$ Dirichlet Problems}

The following is inspired by \cite{Kochubei2023}.
\newline

Let $\Omega\subset X$ be an open subset of a compact $p$-adic analytic manifold
$(X,\mathcal{A})$ with an $O_K$-compatible atlas such that $N(\mathcal{A})$ is connected.
\newline

The \emph{height-$k$-boundary} of an open set $\Omega\subset X$ is defined as
\[
\partial_k\Omega = \mathset{(x,y)\in E_k\mid  x \in\Omega,\; y \in X \setminus\Omega}\,,
\]
where $k\in\mathbb{N}$ is sufficiently large, $E_k$ is the edge set of 
the height-$k$-nearest neighbour graphon $\kNNh (X)$ associated with $X$, and
$m >> 0$.
The \emph{$k$-boundary} of $\Omega$ is defined as
\[
\delta_k\Omega
=\mathset{y\in X\setminus\Omega\mid\exists\;x\in\Omega\colon(x,y)\in\partial_k\Omega}\,.
\]
Now, define
\[
\closure_k\Omega = \Omega\sqcup\delta_k\Omega
\]
as the \emph{$k$-closure} of $\Omega$. Clearly, $\closure_k\Omega$ is a closed-open subset of $X$.
\newline

Denote the atlas on a clopen $Z\subset X$ induced by $\mathcal{A}$ as
\[
\mathcal{A}\cap Z,
\]
as it consists of restricting all charts $(U,\phi)\in\mathcal{A}$ to $U \cap Z$. This gives $Z$ the
structure of a compact $p$-adic analytic manifold. In this way, $(\closure_k\Omega, \mathcal{A}\cap\closure_k\Omega)$
becomes a compact $p$-adic analytic manifold.
\newline

Define
\[
\Star_k^\Omega(x) = \Star_k(x)\cap\closure_k\Omega\,,
\]
and the $k$-nearest neighbour operator on $\closure_k\Omega$ is
\[
\Delta_{k,\Omega}^\alpha f(x)
=\int_{\Star_{k,\Omega}^\alpha}T_{k,\Omega}^\alpha(x,y)(u(x)-u(y))\,d\mu_\Lambda(y)
\]
with kernel function
\[
T_{k,\Omega}^\alpha(x,y)=\begin{cases}
\frac{\displaystyle\deg_k(y)}{\displaystyle\deg_k(x)}\cdot d_\Lambda(x,y)^{-n\alpha},&y\in\Star_{k}^{\Omega}(x)
\\
0,&\text{otherwise}
\end{cases}
\]
for $x,y\in X$, acting on functions on $u\in\mathcal{D}(\closure_k\Omega)$.
\newline

Further, define the Sobolev spaces of interest as
\begin{align*}
W^{1,2,\alpha}(\closure_k\Omega,\mu_k) 
&= \mathset{f \in L^2(\closure_k\Omega,\mu_k)\mid \norm{\Delta^\alpha_{\Omega,k} f}_{L^2(\closure_k\Omega,\mu_\Lambda)} < \infty}\,,
\\
W_0^{1,2,\alpha}(\closure_k\Omega,\mu_\Lambda) 
&=\mathset{f \in W^{1,2,\alpha}\mid f|_{\delta_k\Omega}=0}
\end{align*}
for $\alpha\in\mathbb{R}$, on which the operator $\Delta^\alpha_{k,\Omega}$ is meant to act.

\begin{prop}[Poincar\'e Inequality, $\kNNh$]\label{PIkNNh}
Let $f\in W^{1,2,\alpha}(\closure_k\Omega,\mu_\Lambda)$.
Then there exists $C > 0$ such that
\[
\norm{f}_{L^2(\closure_k \Omega,\mu_\Lambda)} \le 
C\norm{\Delta^\alpha_{\Omega,k}f}_{L^2(\closure_k\Omega,\mu_\Lambda)}
\]
with $\alpha > 0$ fixed.
\end{prop}

\begin{proof}
This proof uses the eigenbasis expansion of $L^2(\closure_k\Omega,\mu_\Lambda)$ w.r.t.\ $\Delta^\alpha_{\Omega,k}$,
and proceeds exactly as in the proof of Proposition \ref{PoincareIneq}, whereby using the
unboundedness of the eigenvalues of $\Delta^\alpha_{\Omega,k}$ on $L^2(\closure_k\Omega,\mu_\Lambda)$ for the chosen $\alpha > 0$.
\end{proof}

\begin{rem} The approach of Propositions \ref{PoincareIneq} and \ref{PIkNNh} uses the unboundedness condition for eigenvalues given only for $\alpha> 0$. 
A.\ Kochubei has an approach
via Sobolev inequalities in \cite{Kochubei2023}. In the future that approach should also be applied
to operators on compact $p$-adic analytic manifolds.
\end{rem}

Let in the following $H_{\mathbb{R}}(A)$ denote the function space of real-valued functions on a set $A$ corresponding to a given function space $H(A)$.
 The use of a basis of $H(A)$ of complex-valued eigenfunctions
for operators, like e.g.\ wavelets for $p$-adic Laplacians, poses no problem, because in the cases
considered here, it can easily be replaced by an eigenbasis of real-valued functions
for $H(A)$.
\newline

The homogeneous Dirichlet problem can be formulated as follows: Given
$f\in L^2(\Omega)_{\mathbb{R}}$, solve
\begin{align}\label{Dirichlet_kNNh}
\Delta^\alpha_{\Omega,k}u(x) &= f(x),&x\in\Omega
\\\nonumber
u(x) &= 0,&x \in\delta_k\Omega,
\end{align}
where $\Omega\subset X$ is an open subset, and where a solution $u$ is meant to be in the Sobolev space $W_0^{1,2,\alpha}(\closure_k\Omega,\mu_\Lambda)_{\mathbb{R}}$ of real-valued functions. With the Sobolev
inner product, it is a Hilbert space.
\newline

The Dirichlet form $Q_k$, defined as
\[
Q_k(u,v) = \langle\Delta^\alpha_{\Omega,k}u, v\rangle_{L^2(\closure_k\Omega,\mu_\Lambda)}\,,
\]
is a bilinear form on $W_0^{1,2,\alpha}(\closure_k\Omega,\mu_\Lambda)_{\mathbb{R}}$, and allows to prove the existence
and uniqueness of weak solutions of (\ref{Dirichlet_kNNh}) in Theorem 
\ref{DirichletSolution_kNNh} below.

\begin{thm}\label{DirichletSolution_kNNh}
Assume that $\alpha > 0$. Then the homogeneous Dirichlet problem
(\ref{Dirichlet_kNNh}) has a unique weak solution, i.e.\ for all 
$\phi\in W_0^{1,2,\alpha}(\closure_k\Omega,\mu_\Lambda)_0$, there exists
a unique $u\in W_0^{1,2,\alpha}(\closure_k \Omega,\mu_\Lambda)$ such that
\[
\langle\phi, v\rangle_{W_0^{1,2,\alpha}(\closure_k \Omega,\mu_\Lambda)} = 
Q_k(u,v)
\]
for all $v \in W_0^{1,2,\alpha}(\closure_k\Omega,\mu_\Lambda)$.
\end{thm}

\begin{proof}
First, show that the bilinear form $Q_k$ is both, continuous and coercive,
on the Hilbert space $W_0^{1,2,\alpha}(\closure_k\Omega,\mu_\Lambda)$:
\begin{align*}
\absolute{Q_k(u,v)}&\le\norm{u}_{L^2(\closure_k\Omega,\mu_\Lambda)}
\norm{\Delta^\alpha_{\Omega,k}v}_{L^2(\closure_k\Omega)}
\\
&\le C\norm{\Delta^\alpha_{\Omega,k}u}_{L^2(\closure_k\Omega,\mu_\Lambda)}
\norm{\Delta^\alpha_{\Omega,k}v}_{L^2(\closure_k\Omega,\mu_\Lambda)}
\\
&\le C\norm{u}_{W_0^{1,2,\alpha}(\closure_k\Omega,\mu_\Lambda}\norm{v}_{W_0^{1,2,\alpha}(\closure_k\Omega,\mu_\Lambda)}
\end{align*}
implies that $Q_k$ is continuous. And
\begin{align*}
Q_k(v,v)&=\int_{X}v(x)\Delta^\alpha_{\Omega,k}v(x)\,d\mu_\Lambda(x)
\\
&=\int_X\int_X v(x)T_{\Omega,k}^\alpha(x,y)(u(x)-u(y))\,d\mu_\Lambda(y)\,d\mu_\Lambda(x)
\\
&\ge
\int_X\int_X v(x)T_{\Omega,k}^\alpha(x,y)(u(x)-u(y))\,d\mu_\Lambda(y)\,d\mu_\Lambda(x)
\\
&-\int_X\int_X v(y)T_{\Omega,k}^\alpha(x,y)(u(x)-u(y))\,d\mu_\Lambda(y)\,d\mu_\Lambda(x)
\\
&=\norm{\Delta_{\Omega,k}^{\frac{\alpha}{2}}v}_{L^2(\closure_k\Omega,\mu_\Lambda)}
\\
&\ge C'
\norm{\Delta_{\Omega,k}^\alpha v}_{L^2(\closure_k\Omega,\mu_\Lambda)}
\\
&\ge C\norm{v}_{L^2(\closure_k\Omega,\mu_\Lambda)}
\end{align*}
for some $C,C'>0$, and the last inequality is the Poincar\'e Inequality (Proposition \ref{PIkNNh}). Notice that the second to the last inequality follows from expanding $v$ in the wavelet eigenbasis and Lemma \ref{waveletSpec_kNNh}.
This now implies coerciveness w.r.t.\ the Sobolev norm. The assertion 
follows from the Lax-Milgram Theorem, cf.\ \cite[Corollary 5.8]{Brezis2010}.
\end{proof}

\begin{lem}
The operator $e^{-t\Delta^\alpha_{\Omega,k}}$ acts compactly on $L^2(\closure_k\Omega,\mu_\Lambda)$ for $t > 0$, $\alpha>0$.
\end{lem}

\begin{proof}
Similarly as in Theorem \ref{kNNh_Spectrum}, and by adapting the wavelet eigenvalue formula of Lemma \ref{waveletSpec_kNNh} accordingly, 
it can be seen the operator 
$e^{-t\Delta_{\Omega,k}^\alpha}$ is trace-class on the Hilbert space $L^2(\closure_k\Omega,\mu_\Lambda)$ for $t > 0$,
$\alpha > 0$. Hence, it acts compactly on $L^2(\closure_k\Omega,\mu_\Lambda)$.
\end{proof}

\begin{thm}
The operator $-\Delta^\alpha_{\Omega,k}$ generates a contraction semigroup $e^{-t\Delta_{\Omega,k}}$ with $t\ge0$ on $W^{1,2,\alpha}(\closure_k\Omega,\mu_\Lambda)$, satisfying the Markov property, where $\alpha > 0$, and $k\in\mathbb{N}$ sufficiently large.
\end{thm}

\begin{proof}
Since $-\Delta^{\alpha}_{\Omega,k}$ acts on the Hilbert space $W^{1,2,\alpha}(\closure_k\Omega,\mu_\Lambda)$ and its eigenvalues are bounded from above (being non-positive), it follows that $e^{-t\Delta_{\Omega,k}}$
is a strongly continuous semigroup acting on $W^{1,2,\alpha}(\closure_k\Omega,\mu_\Lambda)$ for $t\ge 0$. It
is also a contraction semigroup, because 
\[
\norm{\int_0^t e^{-\tau\Delta^\alpha_{\Omega,k}}u\,d\tau}_{W^{1,2,\alpha}(\closure_k\Omega,\mu_\Lambda)}
\le t\norm{u}_{W^{1,2,\alpha}(\closure_k\Omega,\mu_\Lambda)}
\]
by the same argument as (\ref{importantIneq}). As outlined in the proof of Theorem \ref{VTMarkov}, this
implies the contraction semigroup property.
\end{proof}

\subsection{Coordinate Laplacians}

Now, using the concept of $K$-analytic frames, one obtains in a natural way an integral structure:

\begin{lem}
A $K$-analytic frame on  a $p$-adic analytic $n$-manifold defines an integral structure on $X$.
\end{lem}

\begin{proof}
It suffices to show that the map
\[
x \mapsto O_K b_1(x) \oplus\dots\oplus O_K b_n(x) 
\]
is locally
constant on $X$, where 
$(b_1,\dots,b_n)$ is a $K$-analytic frame.
But since each $b_i(x)$ is a $K$-
analytic function in $x$, it follows that a change of $x$ in a sufficiently small
neighbourhood leaves the $O_K$-module $O_K b_i(x)\subset K^n$ invariant. This proves
the assertion.
\end{proof}

A $p$-adic analytic manifold $X$ is called \emph{parallelisable}, if there exists a $K$-
analytic frame on $X$. This is equivalent to the tangent bundle $T(X)$ being
trivial.

\begin{lem}
Any compact $p$-adic analytic manifold is parallelisable.
\end{lem}

\begin{proof}
The idea is to use Serre’s Theorem \cite[Théorème (1)]{Serre1965}. Following the
proof in \cite[Lemma 7.5.1]{Igusa2002}, assume that $X$ has a finite atlas whose underlying
sets are compact subsets of $X$. These sets are bi-analytic to disjoint unions of
translates of $\pi^e O_K^n$
for sufficiently large $e\in\mathbb{N}$. As their tangent bundles are
all trivial, it follows that 
$T(X)$ is a trivial vector bundle of rank $n$.
\end{proof}

Observe now that  if $\Lambda(X)$ is an integral structure on $X$, then the assignment
\[
(x,v_x)\to\Lambda(TX)_{(x,v_x)}:=\Lambda(X)_x\times\Lambda(X)_x \subset K^{2n}
\]
defines an integral structure on the $2n$-manifold $T(X)$.
Assuming that $\Lambda(X)$ comes from a frame $(b_1(x),\dots,b_n(x))$, then the coordinate Laplacian
\[
\Delta_w u(x)=\int_X w(b_i(x),b_i(y))(u(x)-u(y))\,d\mu_{\Lambda(X)}(y)
\]
with some kernel function
$w(\xi,\eta)
$ for $\xi,\eta\in T(X)$ defines an operator on $\mathcal{D}(X)$.
This operator is called the
\emph{$i$-th coordinate Laplacian on $X$ with respect to the frame $(b_1(x),\dots,b_n(x))$}.
\newline

For example, the $i$-th coordinate Vladimirov Laplacian:
\[
\Delta_{b_i}^\alpha u(x)=\int_X d_{\Lambda(TX)}(b_i(x),b_i(y))^{-\alpha n}(u(x)-u(y))\,d\mu_{\Lambda(X)}(y)
\]
or the $i$-th coordinate height-$k$-nearest neighbour Laplacian:
\[
\Delta_{k,b_i}^\alpha u(x)=\int_XT_{T\!X,k}^\alpha(b_i(x),b_i(y))(u(x)-u(y))\,d\mu_{\Lambda(X)}(y)
\]
with
\[
T_{T\!X,k}^\alpha(\xi,\eta)=\begin{cases}
\frac{\displaystyle\deg_k(\eta)}{\displaystyle\deg_k(\xi)}\cdot d_{\Lambda(X)}(\xi,\eta)^{-2n\alpha},&
\eta\in\Star_{T\!X,k}(\xi)
\\
0,&\text{otherwise}
\end{cases}
\]
can now be defined accordingly for the frame $(b_1(x),\dots,b_n)$.

\subsection{Elliptic Dirichlet Problem}

The idea of this subsection is to generalise some results about $p$-adic elliptic diffusion operators from \cite{ellipticBVP}. Given an analytic frame $(b_1(x),\dots,b_n(x))$ on a compact $p$-adic analytic $n$-manifold $X$, define the following operator:
\begin{align*}
P(\Delta_k)u&=\sum\limits_{i,j=1}^d\Delta_{k,b_i}^\alpha(a^{ij}\Delta_{k,b_j}^\alpha u)
\end{align*}
for $u\in\mathcal{D}(X)$, using functions
\[
a^{ij}\colon X\to\mathbb{R}
\]
for $i,j=1,\dots,n$
such that 
\[
A(x)=(a^{ij}(x))\in\mathbb{R}^{n\times n}
\]
is symmetric. If $A(x)$ is positive definite for almost all $x \in X$, and the smallest positive eigenvalue of $A$ is always at least $\theta > 0$, then the operator $P(\Delta)$, resp.\ $P(\Delta_k)$ is called \emph{elliptic}.
\newline

The Dirichlet Problem associated with operator $P(\Delta_k)$ is the following:
\begin{align}\label{DirichletProblemElliptic}
P(\Delta_k)u(x)&=f(x),&x\in \Omega\,,
\\\nonumber
u(x)&=0,&x\in\delta_k\Omega\,.
\end{align}
And a function $u\in W_0^{1,2,\alpha}(\closure_k\Omega,\mu_\Lambda)$ is a
\emph{weak solution} of (\ref{DirichletProblemElliptic}), if
\[
B[u,\phi]
:=\int_\Omega\sum\limits_{i,j=1}^d P(\Delta_{\Omega,k}^\alpha)u(x)\phi(x)\,d\mu_\Lambda(x)
=\int_\Omega f(x)\phi(x)\,d\mu_\Lambda(x)
\]
for all $\phi\in W_0^{1,2,\alpha}(\closure_k\Omega,\mu_\Lambda)$.
\newline

Using the adjoint operator $\Delta_{k,b_i}^{\alpha,*}$ of $\Delta^\alpha_{k,b_i}$ for $i=1,\dots,n$, obtain
\[
\int_{\closure_k\Omega}\Delta_{k,b_j}(a^{ij}(x)\Delta_{k,b_i} u(x)\phi(x)\,d\mu_\Lambda(x)
=\int_{\closure_k\Omega}a^{ij}(x)\Delta_{k,b_i}u(x)\Delta_{k,b_j}^{\alpha,*}\phi(x)\,d\mu_\Lambda(x)\,.
\]
Hence,
\begin{align}\label{ellipticForm}
B[u,\phi]=\int_{\closure_k\Omega}\sum\limits_{i,j=1}^n
a^{i,j}(x)\Delta_{k,b_i}u(x)\Delta_{k,b_j}^{\alpha,*}\phi(x)\,d\mu_\Lambda(x)\,.
\end{align}

\begin{thm}[Energy estimates]
Assume that $a^{ij}\in \mathcal{D}(X)$ for $i,j=1,\dots,n$.
Then there exist constants $\alpha,\beta > 0$ such that
\begin{align*}
\absolute{B[u,v]}&\le\alpha\norm{u}_{W_0^{1,2,\alpha}(\closure_k\Omega,\mu_\Lambda)}\norm{v}_{W_0^{1,2,\alpha}(\closure_k\Omega,\mu_\Lambda)}
\\
\beta\norm{u}_{W_0^{1,2,\alpha}(\closure_k\Omega,\mu_\Lambda)}&\le B[u,u]
\end{align*}
for all $u,v\in W_0^{1,2,\alpha}(\closure_k\Omega,\mu_\Lambda)$.
\end{thm}

\begin{proof}
Adapt the proof of \cite[Theorem 5.3]{ellipticBVP} as follows:
\begin{align*}
\absolute{B[u,v]}&\le
\sum\limits_{i,j=1}^n \norm{a^{i,j}}_{L^\infty}
\int_{\closure_k\Omega}\absolute{\Delta_{k,b_i}^\alpha u(x)}\absolute{\Delta_{k,b_j}^{\alpha,*}v(x)}\,d\mu_\Lambda(x)
\\
&\le\alpha\norm{u}_{W_0^{1,2,\alpha}(\closure_k\Omega,\mu_\Lambda)}\norm{v}_{W_0^{1,2,\alpha}(\closure_k\Omega,\mu_\Lambda)}
\end{align*}
for some $\alpha>0$, where the Poincar\'e inequality (Proposition \ref{PoincareIneq}) has been used for both $\Delta_{k,b_i}^\alpha$ and $\Delta_{k,b_i}^{\alpha,*}$. Notice that it is also applicable to the latter operator, because its eigenvalues coincide with those of $\Delta^\alpha_{k,b_i}$ for $i=1,\dots,n$.

\smallskip
The following inequality 
\begin{align*}
\theta&\int_{\closure_k\Omega}
\sum\limits_{i,j}^n\Delta_{k,b_i}^\alpha u(x)\Delta_{k,b_j}^{\alpha,*}u(x)\,d\mu_\Lambda(x)
\\
&\le\int_{\closure_k\Omega}\sum\limits_{i,j=1}^n a^{i,j}(x)\Delta_{k,b_i}u(x)\Delta_{k,b_j}^{\alpha,*}u(x)\,d\mu_\Lambda(x)=B[u,u]
\end{align*}
holds true, because the eigenvalues of the matrix $A(x)=(a^{i,j}(x))$ are locally constant in $x\in X$, and by ellipticity: 
view the terms in the inequality above as coming from a Rayleigh quotient involving $A(x)$ as a linear operator on $W_0^{1,2,\alpha}(X,\mu_\Lambda)$. Using the Poincar\'e Inequality of Proposition \ref{PoincareIneq}, the second inequality now follows.
\end{proof}

\begin{cor}
Assume that the operator $P(\Delta_k)$ is elliptic. Then (\ref{DirichletProblemElliptic}) has a unique weak
solution in $W_0^{1,2,\alpha}(\closure_k\Omega,\mu_\Lambda)$.
\end{cor}

\begin{proof}
This follows from the Lax-Milgram Theorem 
\cite[Corollary 5.8]{Brezis2010}.
\end{proof}
\section{Outlook}


Applications of diffusion on ultrametric
analytic manifolds, and in particular the novel elliptic operators, are envisioned in \emph{locally} hierarchical data. For example, if there are interactions between hierarchical clusters of data, then these clusters are ultrametric, whereas the whole dataset is then in general only locally ultrametric. It can be suggested to further develop the methods of 
\cite{AngelDiss} towards using $p$-adic manifold methods for deep learning.
\newline

The application we have in mind for the results developed in this contribution is the $p$-adic analytic manifold underlying an algebraic variety $X$ defined over $p$-adic number fields. One goal in this context is to extract number-theoretic information about $X$ through the spectral properties of, as well as the properties of the stochastic process related with, diffusion operators on $X$, like the elliptic operators of this article. 
\newline

On the other hand, generalising the theory developed here to the case of locally hierarchical structures with irregular branching entails the task of defining ultrametric manifolds, and then develop a notion of elliptic diffusion operators on them. This would allow to generalise results of \cite{VPZeta}, and enable applications on the transcendental part of the complex $p$-adic points of algebraic varieties, as has been initiated in \cite{SchottkyTrans,HalwasBA}.
And this would extend the extraction of number-theoretic information of an algebraic variety through a spectral and stochastic study of its $p$-adic transcendental points.

\section*{Acknowledgements}
Dear Branko, Many years! To you who enabled significant interactions via inviting me to conferences organised by you and through your research in more or less related fields which inspired my own research.
\newline

The author also acknowledges the question raised by Wilson Z\'u\~niga-Galindo on a $p$-adic conference in 2019 on how to possibly construct diffusion operators on $p$-adic  manifolds. This question inspired the author's research up to this day and is foundational  also to this work. Also acknowledged are frequent discussions with \'Angel M\'oran Ledezma which significantly contributed to this work, too.
\newline

This research is partially supported by the
Deutsche Forschungsgemeinschaft
under grant number 469999674.


\begin{thebibliography}{19} 

\bibitem{hearGenusMumf}
P.E. Bradley. \emph{Heat equations and hearing the genus on p-adic Mum
ford curves via automorphic forms}. Moscow Mathematical Journal,
{\bf 25}(4) (2025), 447--478.

\bibitem{ellipticBVP}
P.E. Bradley. \emph{Boundary Value Problems for $p$-Adic Elliptic Parisi-Zúñiga Diffusion}, Journal of Pseudo-Differential Operators and Applications, \textbf{17} (2026), 1.

\bibitem{diff_AMf_p}
P.E. Bradley. \emph{Chapter 1: Diffusion on $p$-adic Analytic Manifolds}, in: \emph{Compact Ultrametric Analytic Manifolds and Applications in
Number Theory}, ongoing habilitation project, Karls\-ruhe Institute of Technology (2026)

\bibitem{SchottkyTrans}
P.E. Bradley. \emph{Schottky invariant diffusion on the transcendent $p$-adic upper half
plane}, Nonlinear Anal. \textbf{263} (2026), 113947

\bibitem{HearingSerre} 
P.E. Bradley and A. Moran Ledezma. \emph{Hearing the Serre invariant of a compact $p$-adic analytic manifold} (2025), arXiv:2511.20631 [math.NT]

\bibitem{DiffMfp}
P.E. Bradley. \emph{Diffusion operators on p-adic analytic manifolds} (2025), arXiv:2510.22563 [math.AP]

\bibitem{VPZeta}
P.E. Bradley. \emph{Vladimirov-Pearson Operators on $\zeta$-regular Ultrametric Cantor Sets},
Mathematische Nachrichten, \textbf{298} (2025), 3613--4016

\bibitem{Brezis2010}
H. Brezis. \emph{Functional Analysis, Sobolev Spaces and Partial Diﬀerential
Equations}. Springer, New York (2010)

\bibitem{BKL2026}
P. Bürgisser, A. Kulkarni, and A. Lerario, \emph{Nonarchimedean integral
geometry}, Selecta Mathematica, \textbf{32} (2026), 10.

\bibitem{EK2000}
K.-J. Engel and R. Nagel. \emph{One-Parameter Semigroups for Linear Evolution
Equations}. Springer, New York (2000)

\bibitem{HalwasBA}
P. Halwas, \emph{The trees associated with the transcendental $p$-adic numbers}.
Bachelor’s thesis, Karlsruhe Institute of Technology (2025).

\bibitem{Igusa2002}
J.-I. Igusa, \emph{An introduction to the theory of local Zeta functions}, volume 14
of AMS/IP studies in advanced mathematics. American Mathematical
Society, International Press (2002).

\bibitem{XKZ2018}
A. Khrennikov, S. Kozyrev, and W.A. Zúniga-Galindo. \emph{Ultrametric Pseudodifferential Equations and Its Applications}. Encyclopedia of Mathematics and Its Applications, vol. 168. Cambridge University Press (2018).

\bibitem{Kochubei2001}
A.N. Kochubei, \emph{Pseudo-Diﬀerential Equations and Stochastics Over Non-Archimedean Fields}. Marcel Dekker, Inc., New York (2001).

\bibitem{Kochubei2023}
A.N. Kochubei, \emph{The Vladimirov-Taibleson operator: Inequalities,
Dirichlet problem, boundary Hölder regularity}, J. Pseudo-Diﬀer. Oper.
Appl., \textbf{14} (2023), 31.

\bibitem{Kozyrev2004}
S.V. Kozyrev, \emph{$p$-Adic pseudodiﬀerential operators and p-adic wavelets}.
Theoretical and Mathematical Physics, \textbf{138} (2004), 322--332.

\bibitem{AngelDiss}
Á.M. Ledezma, \emph{Ultrametric Spaces: Spectral and Stochastic
Methods with Applications in Science},
submitted PhD thesis, Karlsruhe Institute of Technology (2026).

\bibitem{PRSW2024}
T. Pierce, R. Rajkumar, A. Stine, D. Weisbart, and A.M. Yassine. \emph{Brownian motion in a vector space over a local field is a scaling limit}. Expositiones Mathematicae, 42(6):125607 (2024).

\bibitem{PW2025}
T. Pierce and D. Weisbart. \emph{Brownian motion in the p-adic integers is a
limit of discrete time random walks}. J Stat Phys, 192:104 (2025).

\bibitem{RZ2008}
J.J. Rodr\'{\i}guez-Vega and W.A. Z\'u\~niga-Galindo.
\emph{Taibleson operators, $p$-adic parabolic equations and ultrametric diffusion},
Pacific Journal of Mathematics,
\textbf{237} (2), 
(2008).

\bibitem{Schneider2011}
P. Schneider, \emph{$p$-Adic Lie Groups}. Grundlehren der mathematischen Wissenschaften 344. Springer-Verlag, Berlin Heidelberg (2011)

\bibitem{Serre1965}
J.-P. Serre, \emph{Classiﬁcation des variétés analytiques $p$-adiques compactes}.
Topology, {\bf 3} (1965),409--412.

\bibitem{Serre1992}
J.-P. Serre, \emph{Lie Algebras and Lie Groups. Lectures given at Harvard Uni-
versity}. Lecture Notes in Mathematics 1500. Springer (1992).

\bibitem{Taira1998}
K. Taira. \emph{Brownian motion and index formulas for the de {Rham} complex},
{Mathematical Research 106},
{Berlin},
{Wiley-VCH},
{215 p.},
(1998)


\bibitem{vanKampen2007}
N.G. van Kampen. \emph{Stochastic Processes in Physics and Chemistry}. North-Holland Personal Library, 3rd edition, (2007)

\bibitem{RodriguezDiss}
J.P. Velasquez-Rodriguez.
\emph{Análisis armónico en grupos pro-finitos},  PhD Thesis,
Universidad del Valle (2025)

\bibitem{Rodriguez2025}
J.P. Velasquez-Rodriguez.
\emph{Unitary dual and matrix coefficients of compact nilpotent
$p$-adic Lie groups with dimension $d \le 5$},
Bol. Soc. Mat. Mex., \textbf{31} (2025), 37.

\bibitem{VVZ1994}
Vladimirov V.S., Volovich I.V., and Zelenov E.I. \emph{$p$-adic Analysis and
mathematical physics}. Series on Soviet and East European Mathematics,
1. World Scientific Publishing Co., Inc., River Edge, NJ, (1994).

\bibitem{Weil1982}
A. Weil, \emph{Adeles and Algebraic Groups}. Progress in Mathematics 23.
Birkhäuser, Boston (1982).

\bibitem{Zuniga2025}
W.A. Zúñiga-Galindo, \emph{$p$-Adic Analysis: Stochastic Processes and Pseudo-
Diﬀerential Equations}. De Gruyter, Berlin (2025)
\end{thebibliography}
\end{document}